\documentclass[sn-apa]{sn-jnl}
\usepackage{amsmath}
\usepackage{amsthm}
\usepackage{amsfonts}
\usepackage{amssymb}
\usepackage[version=4]{mhchem}
\usepackage{stmaryrd}
\usepackage{mathrsfs}
\usepackage[scr=rsfso]{mathalfa}

\usepackage{graphicx}
\usepackage{array}
\usepackage{booktabs}
\usepackage{multirow}
\usepackage{adjustbox}
\usepackage{float}
\usepackage{tabularx}
\usepackage{makecell} 
\usepackage{subcaption}

\usepackage[utf8]{inputenc}
\usepackage{xcolor}
\usepackage{textcomp}
\usepackage{listings}
\usepackage{algorithm}
\usepackage{algorithmicx}
\usepackage{algpseudocode}
\usepackage{caption}
\usepackage{enumitem}

\hypersetup{
    colorlinks=true,
    linkcolor=blue,
    citecolor=blue,
    urlcolor=blue
}

\usepackage[english]{babel}
\usepackage[T1]{fontenc}
\usepackage{csquotes}

\theoremstyle{thmstyleone}
\newtheorem{theorem}{Theorem}

\theoremstyle{thmstyletwo}

\newtheorem{remark}{Remark}
\theoremstyle{thmstylethree}

\numberwithin{equation}{section}
\DeclareUnicodeCharacter{03B6}{\ifmmode\zeta\else{$\zeta$}\fi}
\DeclareUnicodeCharacter{03C4}{\ifmmode\tau\else{$\tau$}\fi}
\DeclareUnicodeCharacter{03A3}{\ifmmode\Sigma\else{$\Sigma$}\fi}
\DeclareUnicodeCharacter{202F}{\ifmmode\ \else{ }\fi}
\DeclareUnicodeCharacter{200B}{\ifmmode\Sigma\else{$\Sigma$}\fi}
\title{A Microeconomic Finance Model with a Multi-Asset Market and a Multi-Investor Heterogeneous Groups}
\author*[1]{\fnm{Mario} \sur{Cavani}}\email{mcavani@udo.edu.ve}
\affil*[1]{\orgdiv{Department of Mathematics}, \orgname{Universidad de Oriente}, \orgaddress{\street{Av. Universidad}, \city{Cuman\'a}, \postcode{6101}, \country{Venezuela}}}
\date{}

\begin{document}
\maketitle
\begin{abstract}
\indent We present a mathematical model of a market with $m$ shares traded across $n$ investor groups, each one with similar motivations and trading strategies. The market of each asset consists of a fixed amount of cash and shares (no additions are allowed over time, so the system is closed), and the trading groups are influenced by trend and valuation motivations when buying or selling each asset, but follow a strategy where the purchase of one asset depends on the price of another, while the sale does not. Using these assumptions and basic microeconomic principles, the mathematical model is derived using a dynamic systems approach. We analyze the stability of the model's equilibrium points and determine the parameter conditions for such stability. First, we show that all equilibria are stable in the absence of a clear emphasis on trend-based valuation for each share. Secondly, for systems where the trading group prioritizes the valuation of each stock and the trend of the other for trading purposes, we establish stability conditions and demonstrate with numerical examples that when instability occurs, it manifests as price oscillations in the stocks. Furthermore, we argue for the existence of periodic solutions via a Hopf bifurcation, taking the momentum coefficient as the bifurcation parameter. Finally, we present examples and numerical simulations to support and expand upon the analytical results. One finding in economics and finance is the existence of cyclical behavior in the absence of exogenous factors, as determined by the momentum coefficient. In particular, a stable equilibrium price becomes unstable as trend-based trading increases.

\end{abstract}
\section{Introduction}
Recent advances in mathematical finance have highlighted the importance of incorporating both multiple asset classes and heterogeneous investor groups when modeling market dynamics. This paper extends the two-asset model presented by \citet{Bulut2019} to an $m$-asset model, and incorporates multiple investor groups with different trading strategies, building upon the multi-group framework developed by \citet{CaginalpDeSantis2011}. The resulting model captures complex interactions between different asset classes and investor types, providing a more comprehensive framework for analyzing market stability, bifurcations, and cyclic behavior.
In this paper, we follow in the footsteps of these earlier models, but we focus on the price dynamics of an $m$-asset market. We still assume that the number of shares of each asset and the amount of cash are constant over time, but each trader's trading strategy now depends on all the stock prices, which couple their resulting dynamics. The model is an extension of the models derived for a single asset market system given in \citep{CaginalpBalenovich1994,caginalp1999asset}. It is derived by assuming similar conditions, including the finiteness of assets, so that unlimited arbitrage is not possible. It is assumed that there are $N_j^{(i)}$ shares of stock $i$, and $M_j$ units of cash for each group of investor $j$ in the system. These investors follow a trading strategy in which the buying of an asset depends on the other asset’s prices, while the selling does not. With respect to this trading strategy, each investor group has preference functions for a stock that are influenced by price momentum and discount from fundamental value. Using basic microeconomic principles, we derive a complete system of first-order non-linear differential equations. We present the equilibrium stability analysis of the model for several cases and use numerical simulations to support and extend the analytical results

\section{Literature Review}
Recent literature in financial mathematics has emphasized several key developments:

\begin{enumerate}
  \item \textit{Multi-Asset Models}: The foundational work by \citet{Bulut2019}  established a deterministic framework for two-asset markets and one homogeneous group of investors, demonstrating how coupling between assets can lead to complex dynamics, including Hopf bifurcations. Subsequent work by \citet{DeSantisSwigon2018} extended this to analyze dynamics in wealth distribution. \citet{he2018time} develop a continuous‑time heterogeneous‑agent financial market model of multi‑assets traded by fundamental and momentum investors; they provide a mechanism for generating time‑varying dominance between fundamental and non‑fundamental trading, showing that investment constraints lead to the coexistence of a locally stable fundamental steady state and a locally stable limit cycle around the fundamental, characterised by a generalised Hopf bifurcation; this bistability provides a mechanism for market prices to switch stochastically between two persistent but very different market states,  a phenomenon that can explain the seemingly contradictory coexistence of efficient‑market behaviour and price momentum over different time periods. The model there also generates spillover effects in both momentum and volatility, market booms, crashes, and correlation reduction due to cross‑sectional momentum trading; empirical evidence from the U.S. market supports the main findings. Other recent articles on the topic that give an important vision of the multi-asset topic are \citep{al2022novel,fahim2024derivation}.

  \item \textit{Heterogeneous Agent Models}: \citet{CaginalpDeSantis2011} introduced a rigorous mathematical framework for markets with multiple investor groups, each with distinct trading strategies. \citet{desantis2023asset} analyses a deterministic asset flow model for a homogeneous group of investors with identical trading sentiments,  with natural modifications of the response functions, the model can be represented by a linear dynamical system across the entire state space, yet it still exhibits all types of price behaviour (overshoots, oscillations, bubbles) as the nonlinear version. The paper presents a limiting case that corresponds to fast trading, appropriate for modern electronic markets, and provides examples illustrating the differences in price response to variability in intrinsic value or liquidity. \citet{li2022continuous} develop a continuous heterogeneous agent model (HAM) where agents are identified by their wealth and ex‑post price expectations, relying on weaker behavioural assumptions than traditional models; the joint evolution equations for stock price and wealth distribution are derived under a bounded‑arbitrage condition, an endogenous generating mechanism for agent types arises from a splitting‑mixing property, linking the continuous HAM to classical HAMs with finite types; numerical simulations show that wealth dynamics positively associate the weight of each generated agent type with their average investment performance, while price dynamics exhibit period‑doubling and chaotic bifurcations. The authors there, introducing random shocks, the model can simultaneously generate volatility clustering and long‑range dependence,  two stylized facts that are often difficult to replicate in deterministic models. \citet{chen2023evolutionary} use an agent‑based model to investigate the effects of homogeneous and heterogeneous investment strategies and reference points on price movements; the key insight is that markets flooded with homogeneous investors experience large price fluctuations, while markets with heterogeneous investors exhibit only moderate fluctuations; moreover, the coexistence of different types of investment strategies can help both refrain from large price fluctuations and avoid no‑trading states. Also, the papers \citep{caginalp2020nonlinear,caginalp2021stochastic,schnetzer2022evolutionary,chan2022investor} are interesting ones and related to our research.

  \item \textit{Behavioral Factors}: The model here was built upon earlier behavioral finance work by \citep{shefrin2008behavioral,DanielHirshleiferSubrahmanyam1998}, also the work by \citet{MerdanAlisen2011,caginalp2008dynamics} has further developed the incorporation of psychological factors like momentum and overreaction into deterministic models. The book of \citet{shefrin2008behavioral} offers a wide vision of this topic.

  \item \textit{Empirical Validation}: Experimental studies by \citet{caginalp2000momentum} and \citet{smith1988} have provided empirical support for many of the nonlinear phenomena predicted by these models. \citet{zhou2022continuous} propose a continuous heterogeneous‑agent model (HAM) that applies to a general $n > 1$ risky assets under market matching friction; every investor is identified by continuous variables representing personal characteristics, and the distribution of these variables captures the market heterogeneity. Next, by introducing market friction via a matching‑based pricing mechanism, the authors derive a pricing equation that extends the classical Capital Asset Pricing Model (CAPM) and specifies how bounded rationality and heterogeneity drive asset prices away from their CAPM values; the model provides a rolling‑window‑based maximum-likelihood algorithm for calibration, which then yields a forecast tool for future stock returns and an adaptive portfolio-construction strategy. 

  \item \textit{Recent Advances}: \citet{cordoni2024instabilities} extend the classic two‑agent, one‑asset market impact game to a general setting with multiple assets and multiple agents; they study the Nash equilibria under transient price impact and quadratic transaction costs, finding that a larger number of assets, greater heterogeneity in trading skills, and a more complex cross‑impact matrix (i.e., multiple liquidity factors) make the market more prone to large oscillations and instability. \citet{dieci2018steady} provide a full analytical treatment of a multi‑asset market model in which speculators can switch between two risky assets and one safe asset.; a four‑dimensional nonlinear map drives the dynamics and may undergo a transcritical, flip, or Neimark‑Sacker bifurcation,  the first bifurcation is associated with an undervaluation of the risky assets, while the latter two may trigger complex endogenous dynamics; the authors first study a simpler two‑dimensional setup to build intuition, then extend to the full multi‑asset case. We mention the works \citep{caginalp2021stochastic} that are related to our research. \citet{Akhter2024} have used a similar model for one asset, one investor related to Bitcoin.

\end{enumerate}

\section{The Mathematical Model}

\subsection{Market Structure}
Consider a financial market involving $m$  assets with prices \( P^{(i)}(t) \), \( i = 1,\ldots,m \), traded within an heterogeneous group of $n$ investors, i.e., investors who share their trading strategies and preferences in each one of the $j-th$ group.  It is assumed that each group of investors has a fixed amount of cash \( M_j \) and shares \( N_j^{(i)} \) for each group \( j \) units of stock $i$. We assume that the trading group follows a strategy in which the buying of an asset depends on the other asset’s price, while the selling does not. With respect to our assumption on the trading strategy.

\subsection{Investor Sentiment Dynamics}
For each asset \( i \) and group \( j \), we define sentiment variables:

\begin{equation}\label{trend_sentiment}
\frac{d\zeta_{1,j}^{(i)}}{dt} = c_{1,j}^{(i)}q_{1,j}^{(i)}\frac{1}{P^{(i)}}\frac{dP^{(i)}}{dt} - c_{1,j}^{(i)}\zeta_{1,j}^{(i)}
\end{equation}

\begin{equation}\label{value_sentiment}
\frac{d\zeta_{2,j}^{(i)}}{dt} = c_{2,j}^{(i)}q_{2,j}^{(i)}\left(1 - \frac{P^{(i)}}{P_a^{(i)}}\right) - c_{2,j}^{(i)}\zeta_{2,j}^{(i)}
\end{equation}

where, for $i=1,\dots, m$, $j=1,\dots, n$, \( \zeta_{1,j}^{(i)} \) is the trend-based sentiment, \( \zeta_{2,j}^{(i)} \)is the value-based sentiment, $ c_{1,j}^{(i)}$ and $c_{2,j}^{(i)}$ represent the time scales and $q_{1,j}^{(i)}$ and $q_{2,j}^{(i)}$ characterize magnitudes for the investors preferences for asset $i$ \citep{caginalp2007asset,caginalp1999asset}. In these equations, $P_{a}^{(i)}(t)$ denotes the fundamental value, while $P^{(i)}(t)$ is the trading price of the asset $i$. 

\subsection{General Definition of the Transition Rate Functions}
The \textit{transition rate functions} are the core behavioral rules governing how investors allocate cash to assets (buying) and convert assets back to cash (selling). In the model, they are defined as probabilities (or rates) that depend on investors' sentiments, which combine trend-based and value-based components. The buying and selling rates for group \( j \) depend on all sentiment variables. With respect to our assumption on the trading strategy, we define transition rate functions as follows:\\
\textit{Buying rate}:\\
The probability that one unit of cash held by the group $j$ is used to purchase one unit of the asset $i$ per unit time,
\[
k_j^{(i)} = k_j^{(i)}(\{\zeta_{1,k}^{(l)}. l=1, \ \ldots,m,\ k=1,\ldots,n\},\{\zeta_{2,k}^{(l)}, \ l=1,\ldots,m,\ k=1,\ldots,n\}). 
\]
\textit{Selling rate}:\\
The probability that one unit of asset $i$ held by the group $j$ is sold (converted to cash) per unit time,
\[ \tilde{k}_j^{(i)} = \tilde{k}_j^{(i)}(\{\zeta_{1,k}^{(l)}, \ l=1,\ldots,m,\ k=1,\ldots,n\},\{\zeta_{2,k}^{(l)}, \ l=1,\ldots,m,\ k=1,\ldots,n\})
\]
The \textit{transition rate functions} are subject to the following \textit{constraints}:
\[
0 \le k_j^{(i)} \le 1,\quad 0 \le \tilde{k}_j^{(i)} \le 1,\quad \sum_{i=1}^m k_j^{(i)} \le 1.
\]
\textit{Extension to Multiple Assets and Heterogeneous Investor Groups}

Here we extend the definition of the transition rate functions according to \citet{Bulut2019}. For \( m \) assets and \( n \) groups, the rates can be generalized. A typical form for a \textit{trend‑following} investor is:

\begin{equation}\label{k-trend}
k_j^{(i)} = a_j^{(i)} + b_j^{(i)} \tanh\!\left(\sum_{l=1}^m \alpha_{j,l}^{(i)} \zeta_{1,j}^{(l)} + \sum_{l=1}^m \beta_{j,l}^{(i)} \zeta_{2,j}^{(l)}\right),
\end{equation}

where \( a_j^{(i)}, b_j^{(i)} \) are constants chosen so that \( k_j^{(i)} \in [0,1] \) and the sum over \( i \) does not exceed 1.  \\
For \textit{value‑based} investors, we often use a linear form:

\begin{equation}\label{k-value}
{k}_j^{(i)} = \max\!\left(0,\; \min\!\left(1,\; c_j^{(i)} - d_j^{(i)}\left(\frac{P^{(i)}}{P_a^{(i)}} - 1\right)\right)\right),
\end{equation}

which makes buying increase when the asset is undervalued (\(P^{(i)} < P_a^{(i)}\)).
\begin{remark}
According to \citet{Bulut2019} the notation adopted is:
\begin{center}
 \( k_j^{(i)} \) : \textit{buying} rate (cash, thet convert in asset \( i \)),

 \( \tilde{k}_j^{(i)} \) : \textit{selling} rate (asset \( i \) that convert in  cash).
\end{center}
Both rates can depend on sentiment, but the functional forms are chosen so that \textit{buying} rates depend on both assets’ sentiments. In contrast, \textit{selling} rates depend only on the sentiment of the asset being sold. This asymmetry reflects the assumption of the trading strategy: "buying of an asset depends on the other asset's price while the selling does not."\\
Thus, \textit{value‑based} strategies can appear in either buying or selling rates depending on the investor’s behavior. For example:\\
\textit{Value‑based buying}: \( k_j^{(i)} \) decreases when \( P^{(i)} > P_a^{(i)} \) (asset overvalued → less buying).\\
\textit{Value‑based selling}: \( \tilde{k}_j^{(i)} \) increases when \( P^{(i)} > P_a^{(i)} \) (asset overvalued → more selling).\\
So, \textit{value‑based} is not tied exclusively to the \textit{tilde notation}; it can appear in either rate, or both, depending on the investor group’s behavior. The key is that the rates are functions of sentiment, and value‑based sentiment \( \zeta_{2,j}^{(i)} \) enters that function.
\end{remark}

\subsection{Price Dynamics and Wealth Distribution}
The \textit{price} of each asset is determined by adjustment to the excess demand \citep{henderson1971microeconomic,caginalp1999asset}, i.e.,

\begin{equation}\label{price_eq}
\tau_i\frac{1}{P^{(i)}}\frac{dP^{(i)}}{dt} = \frac{\sum_{j=1}^n k_j^{(i)}M_j}{\sum_{j=1}^n \tilde{k}_j^{(i)}N_j^{(i)}P^{(i)}} - 1
\end{equation}

The \textit{wealth distribution} among groups evolves as:

\begin{equation}\label{wealth_eq}
W_j(t) = \frac{M_j(t) + \sum_{i=1}^m N_j^{(i)}(t)P^{(i)}(t)}{\sum_{k=1}^n \left[M_k(t) + \sum_{i=1}^m N_k^{(i)}(t)P^{(i)}(t)\right]}
\end{equation}

\section{Stability Analysis}
The equilibrium points \( E = (P_{eq}^{(i)}, \{\zeta_{k,j,eq}^{(i)}\}) \) must satisfy:
\begin{equation}\label{price_equal_zero}
   \frac{dP^{(i)}}{dt}  =  0, 
\end{equation}
\begin{equation}\label{sent_equal_zero}
    \frac{d\zeta_{k,j}^{(i)}}{dt}  =  0,
\end{equation}

$i =1,\dots, m;  k=1,\dots, m;  j =1,\dots, n$.\\

We need to determine the equilibrium points of the system. Given the model, for each asset $i$:

$$  \frac{dP^{(i)}}{dt}= \frac{P^{(i)}}{τ_i} \left(\frac{Σ_j k_j^{(i)} M_j}{Σ_j \tilde{k}_j^{(i)} N_j^{(i)} P^{(i)}}\right) - 1 ) = 0.$$

Since $P^{(i)} > 0$, equilibrium requires

$$\frac{Σ_j k_j^{(i)} M_j}{Σ_j \tilde{k}_j^{(i)} N_j^{(i)} P^{(i)}} = 1$$

that implies

$$Σ_j k_j^{(i)} M_j = Σ_j \tilde{k}_j^{(i)} N_j^{(i)} P^{(i)},$$

which is the price equilibrium condition.

For sentiment dynamics, we have:

$$\frac{dζ_{1,j}^{(i)}}{dt} = c_{1,j}^{(i)} q_{1,j}^{(i)} \frac{1}{P^{(i)}}\frac{dP^{(i)}}{dt} - c_{1,j}^{(i)} ζ_{1,j}^{(i)} = 0.$$

At equilibrium $dP^{(i)}/dt = 0$, so

$$-c_{1,j}^{(i)} ζ_{1,j}^{(i)} = 0 \qquad\text{implies,}\qquad ζ_{1,j}^{(i)} = 0$$

for all $i,j$.

$$\frac{dζ_{2,j}^{(i)}}{dt} = c_{2,j}^{(i)} q_{2,j}^{(i)} \left(1 -\frac{ P^{(i)}}{P_a^{(i)}}\right) - c_{2,j}^{(i)} ζ_{2,j}^{(i)} = 0$$

which implies,

$$ζ_{2,j}^{(i)} = q_{2,j}^{(i)} \left(1 - \frac{P^{(i)}}{P_a^{(i)}}\right).$$

So the sentiment variables at equilibrium are determined by the prices:
$$ζ_{1,j}^{(i)}=0, \qquad  ζ_{2,j}^{(i)} = q_{2,j}^{(i)} \left(1 -\frac{P^{(i)}}{P_a^{(i)}}\right).$$
The transition rates $k_j^{(i)}$ and $\tilde{k}_j^{(i)}$ depend on all sentiments, so they become functions of the prices.

Thus, equilibrium prices satisfy the price equation with sentiments expressed in terms of prices. This is a system of m equations for m prices.

To get an explicit equilibrium point, we need specific functional forms for $k_j^{(i)}$ and $\tilde{k}_j^{(i)}$, which are functions of sentiments. The model assumes the formula given in (\ref{k-trend}) and (\ref{k-value}).  With these explicit forms, the prices are given explicitly. However, we can state the equilibrium conditions.

The equilibrium point is: For each $i$, choose $P_{eq}^{(i)}$ solving

$$\Sigma_j {k}_j^{(i)}(0, q_{2,j}^{(i)}) \left(1 -\frac{P_{eq}^{(i)}}{P_a^{(i)}}\right) M_j = \Sigma_j \tilde{k}_j^{(i)}(0, q_{2,j}^{(i)}) \left(1 -\frac{P_{eq}^{(i)}}{P_a^{(i)}}\right) N_j^{(i)} P^{(i)}.$$

Thus

$$ζ_{1,j}^{(i)} = 0, \qquad  ζ_{2,j}^{(i)} = q_{2,j}^{(i)} \left(1 -\frac{P_{eq}^{(i)}}{P_a^{(i)}}\right).$$

Often, a natural equilibrium is when all prices equal fundamental values: $P^{(i)} = P_a^{(i)}$. Then $ζ_{2,j}^{(i)}=0$. If the rates also satisfy some symmetry, then the price equation may hold. For instance, if $k_j^{(i)}(0,0)$ and $\tilde{k}_j^{(i)}(0,0)$ are such that

$$Σ_j k_j^{(i)} M_j = Σ_j \tilde{k}_j^{(i)} N_j^{(i)} P_a^{(i)}.$$

This could be ensured by calibration. So one equilibrium is $P^{(i)} = P_a^{(i)}, ζ_1=0, ζ_2=0$.

Thus, a simple explicit equilibrium is the fundamental equilibrium: prices at fundamental values and all sentiments are zero. That is the most common baseline.

Therefore answer: an explicit equilibrium point is given by $P^{(i)} = P_a^{(i)}$ for all $i$, $ζ_{1,j}^{(i)} = 0$, $ζ_{2,j}^{(i)} = 0$ for all $i,j$, provided that the transition rates satisfy

$$Σ_j k_j^{(i)}(0,0) M_j = Σ_j \tilde{k}_j^{(i)}(0,0) N_j^{(i)} P_a^{(i)}.$$

So, an explicit equilibrium point for the system is the \textit{fundamental equilibrium}, where all asset prices equal their fundamental values, and all sentiment variables are zero:

\[
P^{(i)}_{\text{eq}} = P_a^{(i)}, \qquad
\zeta_{1,j}^{(i)} = 0, \qquad
\zeta_{2,j}^{(i)} = 0 \quad \text{for all } i = 1,\dots,m,\; j = 1,\dots,n.
\]

Therefore, the equilibrium conditions are:

\begin{enumerate}[label=\roman*., align=left, leftmargin=*]
  \item\textit{Price Dynamics}:\\

\(\frac{dP^{(i)}}{dt} = 0\) requires,
\[
\frac{\sum_{j} k_j^{(i)} M_j}{\sum_{j} \tilde{k}_j^{(i)} N_j^{(i)} P^{(i)}} = 1.
\]
At the fundamental point, the transition rates are evaluated at zero sentiments:\\
\(k_j^{(i)} = k_j^{(i)}(0,0)\), \(\tilde{k}_j^{(i)} = \tilde{k}_j^{(i)}(0,0)\).\\
The equilibrium holds if the parameters satisfy
\[
\sum_{j} k_j^{(i)}(0,0) M_j = \sum_{j} \tilde{k}_j^{(i)}(0,0) N_j^{(i)} P_a^{(i)}.
\]
This represents a natural calibration condition.\\

  \item \textit{Sentiment Dynamics}:\\

\(\frac{d\zeta_{1,j}^{(i)}}{dt} = -c_{1,j}^{(i)} \zeta_{1,j}^{(i)} = 0\) gives \(\zeta_{1,j}^{(i)} = 0\).\\
\(\frac{d\zeta_{2,j}^{(i)}}{dt} = c_{2,j}^{(i)} q_{2,j}^{(i)}\left(1 - \frac{P^{(i)}}{P_a^{(i)}}\right) - c_{2,j}^{(i)} \zeta_{2,j}^{(i)} = 0\) yields \(\zeta_{2,j}^{(i)} = q_{2,j}^{(i)}\left(1 - \frac{P^{(i)}}{P_a^{(i)}}\right)\), which is zero when \(P^{(i)} = P_a^{(i)}\).
\end{enumerate}

Thus, the fundamental equilibrium is a fixed point of the system. Other equilibria may exist depending on the specific functional forms of the transition rates, but this is the simplest explicit equilibrium.

\subsection{Calculation of the Jacobian}
We derive the Jacobian matrix for the multi-asset, multi-group model described in the paper. The system consists of \(m\) price variables \(P^{(i)}\) and, for each asset \(i\) and investor group \(j\), two sentiment variables \(\zeta_{1,j}^{(i)}\) and \(\zeta_{2,j}^{(i)}\). The total number of state variables is \(N = m + 2mn\).

Define the state vector \(\mathbf{X}\) as:

\[
\mathbf{X} = \big( P^{(1)},\dots,P^{(m)},\; \zeta_{1,1}^{(1)},\dots,\zeta_{1,n}^{(m)},\; \zeta_{2,1}^{(1)},\dots,\zeta_{2,n}^{(m)} \big)^T.
\]

The equations of the ODE system are:\\

\textit{Price dynamics} (for each asset \(i\)):

\[
\frac{dP^{(i)}}{dt} = F_i(\mathbf{X}) = \frac{P^{(i)}}{\tau_i} \left( \frac{S_i}{T_i} - 1 \right),
\]

where

\[
S_i = \sum_{j=1}^n k_j^{(i)} M_j, \qquad T_i = \sum_{j=1}^n \tilde{k}_j^{(i)} N_j^{(i)} P^{(i)}.
\]

The transition rates \(k_j^{(i)}\) and \(\tilde{k}_j^{(i)}\) are functions of all sentiment variables \(\{\zeta_{1,k}^{(l)},\zeta_{2,k}^{(l)}\}\).

\textit{Sentiment dynamics} (for each asset \(i\) and group \(j\)):

\[
\frac{d\zeta_{1,j}^{(i)}}{dt} = G_{i,j}(\mathbf{X}) = c_{1,j}^{(i)} q_{1,j}^{(i)} \frac{1}{P^{(i)}} \frac{dP^{(i)}}{dt} - c_{1,j}^{(i)} \zeta_{1,j}^{(i)},
\]

\[
\frac{d\zeta_{2,j}^{(i)}}{dt} = H_{i,j}(\mathbf{X}) = c_{2,j}^{(i)} q_{2,j}^{(i)} \left(1 - \frac{P^{(i)}}{P_a^{(i)}}\right) - c_{2,j}^{(i)} \zeta_{2,j}^{(i)}.
\]

The Jacobian \(\mathbf{J}\) is an \(N\times N\) matrix with block structure:

\[
\mathbf{J} = \begin{pmatrix}
\mathbf{A} & \mathbf{B} \\
\mathbf{C} & \mathbf{D}
\end{pmatrix},
\]

where:

\begin{itemize}
  \item[] \(\mathbf{A}\) (size \(m\times m\)) contains derivatives of price dynamics with respect to prices.
  \item[] \(\mathbf{B}\) (size \(m\times 2mn\)) contains derivatives of price dynamics with respect to sentiment variables.
  \item[] \(\mathbf{C}\) (size \(2mn\times m\)) contains derivatives of sentiment dynamics with respect to prices.
  \item[] \(\mathbf{D}\) (size \(2mn\times 2mn\)) contains derivatives of sentiment dynamics with respect to sentiment variables.
\end{itemize}

We compute each block explicitly:

Block \( \mathbf{A}\): \(\partial F_i / \partial P^{(k)}\).\\
For \(i\neq k\), \(F_i\) does not depend explicitly on \(P^{(k)}\), so \(\partial F_i/\partial P^{(k)} = 0\). For \(i=k\):

\[
\frac{\partial F_i}{\partial P^{(i)}} = -\frac{1}{\tau_i}.
\]

Thus

\[
\mathbf{A} = \operatorname{diag}\left(-\frac{1}{\tau_1}, \dots, -\frac{1}{\tau_m}\right).
\]

Block \(\mathbf{B}\): \(\partial F_i / \partial \xi\), where \(\xi\) is a sentiment variable.\\
Let \(\xi\) be either \(\zeta_{1,k}^{(l)}\) or \(\zeta_{2,k}^{(l)}\). Then

\[
\frac{\partial F_i}{\partial \xi} = \frac{P^{(i)}}{\tau_i} \cdot \frac{1}{T_i^2} \left( T_i \frac{\partial S_i}{\partial \xi} - S_i \frac{\partial T_i}{\partial \xi} \right),
\]

where

\[
\frac{\partial S_i}{\partial \xi} = \sum_{j=1}^n M_j \frac{\partial k_j^{(i)}}{\partial \xi}, \qquad
\frac{\partial T_i}{\partial \xi} = \sum_{j=1}^n N_j^{(i)} P^{(i)} \frac{\partial \tilde{k}_j^{(i)}}{\partial \xi}.
\]

These partial derivatives of \(k_j^{(i)}\) and \(\tilde{k}_j^{(i)}\) with respect to the sentiment variables are model‑dependent. In general, they are non‑zero for any \(\xi\) because the rates depend on all sentiments.

Block \(\mathbf{C}\): derivatives of sentiment dynamics with respect to prices.\\
For \(\zeta_{1,j}^{(i)}\) (denoted by \(x\)):

\[
\frac{\partial G_{i,j}}{\partial P^{(k)}} = c_{1,j}^{(i)} q_{1,j}^{(i)} \left( \frac{1}{P^{(i)}} \frac{\partial F_i}{\partial P^{(k)}} - \frac{1}{(P^{(i)})^2} F_i \, \delta_{ik} \right).
\]

Since \(\partial F_i/\partial P^{(k)} = -\delta_{ik}/\tau_i\), we obtain

\[
\frac{\partial G_{i,j}}{\partial P^{(k)}} = c_{1,j}^{(i)} q_{1,j}^{(i)} \left( -\frac{\delta_{ik}}{\tau_i P^{(i)}} - \frac{F_i}{(P^{(i)})^2} \delta_{ik} \right) = -c_{1,j}^{(i)} q_{1,j}^{(i)} \frac{\delta_{ik}}{P^{(i)}} \left( \frac{1}{\tau_i} + \frac{F_i}{P^{(i)}} \right).
\]

For \(\zeta_{2,j}^{(i)}\):

\[
\frac{\partial H_{i,j}}{\partial P^{(k)}} = -c_{2,j}^{(i)} q_{2,j}^{(i)} \frac{1}{P_a^{(i)}} \delta_{ik}.
\]

Thus \(\mathbf{C}\) is composed of these entries.

Block \(\mathbf{D}\): derivatives of sentiment dynamics with respect to sentiment variables.\\
First, for \(\zeta_{1,j}^{(i)}\) (denoted \(x_{i,j}\)):

\[
\frac{\partial G_{i,j}}{\partial \xi} = c_{1,j}^{(i)} q_{1,j}^{(i)} \frac{1}{P^{(i)}} \frac{\partial F_i}{\partial \xi} - c_{1,j}^{(i)} \frac{\partial x_{i,j}}{\partial \xi},
\]

where \(\partial x_{i,j}/\partial \xi = 1\) if \(\xi = x_{i,j}\) and \(0\) otherwise.

For \(\zeta_{2,j}^{(i)}\) (denoted \(y_{i,j}\)):

\[
\frac{\partial H_{i,j}}{\partial \xi} = -c_{2,j}^{(i)} \frac{\partial y_{i,j}}{\partial \xi},
\]

so the only non‑zero term is \(-c_{2,j}^{(i)}\) when \(\xi = y_{i,j}\).

Thus \(\mathbf{D}\) has the structure:

\begin{itemize}
  \item[] Diagonal entries for \(\zeta_{1,j}^{(i)}\): \(-c_{1,j}^{(i)}\) plus the contribution from the \(F_i\) derivative.
  \item[] Off‑diagonal entries from the \(\frac{\partial F_i}{\partial \xi}\) terms (which couple all sentiments through the rates).
  \item[] Diagonal entries for \(\zeta_{2,j}^{(i)}\): \(-c_{2,j}^{(i)}\); all other entries in those rows are zero.
\end{itemize}

Therefore, the blocks of the Jacobian matrix are given by:

\[
\mathbf{A} = \operatorname{diag}\left(-\frac{1}{\tau_i}\right),
\]

\[
\mathbf{B}_{i,\xi} = \frac{P^{(i)}}{\tau_i T_i^2} \left( T_i \sum_j M_j \frac{\partial k_j^{(i)}}{\partial \xi} - S_i \sum_j N_j^{(i)} P^{(i)} \frac{\partial \tilde{k}_j^{(i)}}{\partial \xi} \right),
\]

\[
\mathbf{C}_{(i,j),\ k} =
\begin{cases}
-c_{1,j}^{(i)} q_{1,j}^{(i)} \frac{\delta_{ik}}{P^{(i)}} \left( \frac{1}{\tau_i} + \frac{F_i}{P^{(i)}} \right), & \text{for } \zeta_{1,j}^{(i)} \text{ row}, \\[4pt]
-c_{2,j}^{(i)} q_{2,j}^{(i)} \frac{\delta_{ik}}{P_a^{(i)}}, & \text{for } \zeta_{2,j}^{(i)} \text{ row},
\end{cases}
\]

\[
\mathbf{D}_{(i,j,\text{type}),\ \xi} =
\begin{cases}
c_{1,j}^{(i)} q_{1,j}^{(i)} \frac{1}{P^{(i)}} \frac{\partial F_i}{\partial \xi} - c_{1,j}^{(i)} \delta_{\xi,\zeta_{1,j}^{(i)}}, & \text{for } \zeta_{1,j}^{(i)} \text{ row}, \\[4pt]
-c_{2,j}^{(i)} \delta_{\xi,\zeta_{2,j}^{(i)}}, & \text{for } \zeta_{2,j}^{(i)} \text{ row}.
\end{cases}
\]

Note that the subscript “type” in the expression for \(\mathbf{D}\) indicates the \textit{type of sentiment variable} that the row corresponds to. Is a shorthand for “depending on whether the sentiment variable is trend (\(\zeta_1\)) or value (\(\zeta_2\))”. In the full Jacobian matrix, the rows of block \(\mathbf{D}\) are ordered by these variables, and the formula uses this distinction to give the correct derivative. Specifically:\\

If the row is associated with a \textit{trend‑based sentiment} \(\zeta_{1,j}^{(i)}\), the first case applies. While If the row is associated with a \textit{value‑based sentiment} \(\zeta_{2,j}^{(i)}\), the second case applies.\\

These expressions define the Jacobian explicitly in terms of the model parameters and the partial derivatives of the transition rates \(k_j^{(i)}\) and \(\tilde{k}_j^{(i)}\) with respect to the sentiment variables. Each block has the following characteristics:

\begin{itemize}
  \item[] \(\mathbf{A} \) captures price dynamics
  \item[] \(\mathbf{D} \) captures sentiment dynamics
  \item[] \( \mathbf{B} \) and \( \mathbf{C} \) represent coupling terms
\end{itemize}

\begin{theorem} (Stability under homogeneous value investing)
Assume that for every asset \(i=1,\dots,m\) and every investor group \(j=1,\dots,n\):

\begin{enumerate}
  \item[] \textit{No trend‑following}:\\
 \(q_{1,j}^{(i)} = 0\) (so \(\zeta_{1,j}^{(i)}\) dynamics are decoupled and decay).
  \item[] \textit{Positive value‑sentiment sensitivity}: \\
The buying rate \(k_j^{(i)}\) increases with the value sentiment \(\zeta_{2,j}^{(i)}\), and the selling rate \(\tilde{k}_j^{(i)}\) decreases with \(\zeta_{2,j}^{(i)}\); i.e.\[
\frac{\partial k_j^{(i)}}{\partial \zeta_{2,j}^{(i)}} > 0,\qquad
\frac{\partial \tilde{k}_j^{(i)}}{\partial \zeta_{2,j}^{(i)}} < 0.
\]
  \item[] \textit{Weak cross‑asset coupling}:\\
 The dependence of the transition rates on sentiments of different assets is sufficiently small (formally, the off‑diagonal blocks in the Jacobian coupling different assets are negligible compared to the diagonal blocks).
\end{enumerate}

Then the \textit{fundamental equilibrium}

\[
P^{(i)}_* = P_a^{(i)},\quad \zeta_{1,j}^{(i)} = 0,\quad \zeta_{2,j}^{(i)} = 0
\]

is \textit{locally asymptotically stable}.
\end{theorem}

\begin{proof}
We work with the full state vector \(\mathbf{X} = (P^{(1)},\dots,P^{(m)},\;\zeta_{1,1}^{(1)},\dots,\zeta_{2,n}^{(m)})^T\) and the Jacobian \(\mathbf{J}\) evaluated at the fundamental equilibrium.

\textit{Block structure of \(\mathbf{J}\)}\\
From earlier derivations, at the equilibrium:

\begin{itemize}
  \item[] \textbf{Price block \(\mathbf{A}\)}:
\[
\mathbf{A} = \operatorname{diag}\left(-\frac{1}{\tau_1},\dots,-\frac{1}{\tau_m}\right).
\]
This block is diagonal with negative entries, hence stable.\\

  \item[] \textbf{Sentiment block \(\mathbf{D}\)}:\\
Because \(q_{1,j}^{(i)}=0\), the \(\zeta_{1,j}^{(i)}\) dynamics simplify to
\[
\frac{d\zeta_{1,j}^{(i)}}{dt} = -c_{1,j}^{(i)}\zeta_{1,j}^{(i)},
\]
so the corresponding part of \(\mathbf{D}\) is diagonal with entries \(-c_{1,j}^{(i)}<0\).

For \(\zeta_{2,j}^{(i)}\), the dynamics are
\[
\frac{d\zeta_{2,j}^{(i)}}{dt} = c_{2,j}^{(i)}q_{2,j}^{(i)}\left(1-\frac{P^{(i)}}{P_a^{(i)}}\right) - c_{2,j}^{(i)}\zeta_{2,j}^{(i)}.
\]
At equilibrium, \(P^{(i)}=P_a^{(i)}\), so the linearisation gives
\[
\frac{\partial}{\partial\zeta_{2,j}^{(i)}}\left(\frac{d\zeta_{2,j}^{(i)}}{dt}\right) = -c_{2,j}^{(i)}.
\]
Moreover, the derivative with respect to \(P^{(i)}\) is \(-c_{2,j}^{(i)}q_{2,j}^{(i)}/P_a^{(i)}\), which will appear in block \(\mathbf{C}\).\\
Thus \(\mathbf{D}\) is a diagonal matrix with entries \(-c_{1,j}^{(i)}\) (for the \(\zeta_{1}\) variables) and \(-c_{2,j}^{(i)}\) (for the \(\zeta_{2}\) variables); all are negative.\\

  \item[] \textbf{Coupling blocks \(\mathbf{B}\) and \(\mathbf{C}\)}:

  \begin{itemize}
    \item[] \(\mathbf{B}\) contains derivatives of price dynamics with respect to sentiment variables.
    \item[] \(\mathbf{C}\) contains derivatives of sentiment dynamics with respect to prices.
  \end{itemize}
\end{itemize}

At equilibrium, the price dynamics satisfy

\[
F_i = \frac{P^{(i)}}{\tau_i}\left(\frac{S_i}{T_i}-1\right) = 0,
\]

so the term \(\frac{F_i}{P^{(i)}}\) in \(\mathbf{C}\) vanishes. Consequently,

\[
\frac{\partial G_{i,j}}{\partial P^{(k)}}\Big|_{\text{eq}} = -c_{1,j}^{(i)}q_{1,j}^{(i)}\frac{\delta_{ik}}{\tau_i P^{(i)}} = 0,
\]

because \(q_{1,j}^{(i)}=0\). Thus the \(\zeta_{1}\) rows of \(\mathbf{C}\) are zero.\\
For the \(\zeta_{2}\) rows,

\[
\frac{\partial H_{i,j}}{\partial P^{(k)}}\Big|_{\text{eq}} = -c_{2,j}^{(i)}q_{2,j}^{(i)}\frac{\delta_{ik}}{P_a^{(i)}},
\]

which is generally non‑zero but appears only in the block coupling \(P^{(k)}\) to \(\zeta_{2,j}^{(i)}\).

\textit{Decoupling via weak cross‑asset assumption}\\
Because of assumption 3 (weak cross‑asset sentiment effects), the Jacobian becomes \textit{block‑diagonal in the asset index} when the equilibrium is considered. More precisely, for each asset \(i\), the dynamics of \(P^{(i)}\), \(\{\zeta_{1,j}^{(i)}\}_{j=1}^n\), \(\{\zeta_{2,j}^{(i)}\}_{j=1}^n\) decouple from those of other assets. The Jacobian therefore decomposes into \(m\) independent blocks, each of size \(1+2n\) (one price plus \(n\) pairs of sentiments). The stability of the whole system is equivalent to the stability of each block.

\textit{Stability of a single‑asset block}\\
For a fixed asset \(i\), the reduced Jacobian \(\mathbf{J}_i\) (evaluated at the equilibrium) has the structure:

\[
\mathbf{J}_i =
\begin{pmatrix}
-\frac{1}{\tau_i} & \mathbf{B}_i^{(1)} & \mathbf{B}_i^{(2)} \\
\mathbf{0} & -c_{1,1}^{(i)} & \mathbf{0} \\
\vdots & \vdots & \ddots \\
\mathbf{0} & \mathbf{0} & -c_{1,n}^{(i)} \\
\mathbf{C}_i & \mathbf{0} & \mathbf{D}_2
\end{pmatrix},
\]

where:

\begin{itemize}
  \item[] \(\mathbf{B}_i^{(1)}\) (size \(1\times n\)) contains derivatives of \(F_i\) with respect to (w.r.t.) \(\zeta_{1,j}^{(i)}\).
  \item[] \(\mathbf{B}_i^{(2)}\) (size \(1\times n\)) contains derivatives of \(F_i\) w.r.t. \(\zeta_{2,j}^{(i)}\).
  \item[] \(\mathbf{C}_i\) is a row vector of length \(n\) with entries \(-\frac{c_{2,j}^{(i)}q_{2,j}^{(i)}}{P_a^{(i)}}\) for \(j=1,\dots,n\).
  \item[] \(\mathbf{D}_2 = \operatorname{diag}(-c_{2,1}^{(i)},\dots,-c_{2,n}^{(i)})\).
\end{itemize}

Because the \(\zeta_{1,j}^{(i)}\) rows have no coupling back to the price (the zero block under the price row), the eigenvalues associated with these variables are simply \(-c_{1,j}^{(i)}<0\) and are stable.

The remaining subsystem (price \(P^{(i)}\) and the \(n\) value sentiments \(\zeta_{2,j}^{(i)}\)) is:

\begin{equation}
\begin{pmatrix}
\dot{P}^{(i)} \\
\dot{\zeta}_{2,1}^{(i)} \\
\vdots \\[2pt]
\dot{\zeta}_{2,n}^{(i)}
\end{pmatrix}=
\begin{pmatrix}
-\frac{1}{\tau_i} & \mathbf{B}_i^{(2)} \\[2pt]
\mathbf{C}_i & \mathbf{D}_2
\end{pmatrix}
\begin{pmatrix}
P^{(i)} \\[2pt]
\zeta_{2,1}^{(i)} \\[2pt]
\vdots \\[2pt]
\zeta_{2,n}^{(i)}
\end{pmatrix}.
\end{equation}

\textit{Stability of the price–value subsystem}\\
We analyse the \((n+1)\times(n+1)\) matrix

\[
\mathbf{M}_i =
\begin{pmatrix}
-\frac{1}{\tau_i} & \mathbf{b}^T \\[2pt]
\mathbf{c} & \mathbf{D}_2
\end{pmatrix},
\]

where \(\mathbf{b}^T = \mathbf{B}_i^{(2)}\) (row vector) and \(\mathbf{c} = \mathbf{C}_i\) (column vector).\\
\(\mathbf{D}_2\) is diagonal with negative entries \(-c_{2,j}^{(i)}\).

The characteristic polynomial of \(\mathbf{M}_i\) is

\[
\det(\lambda\mathbf{I} - \mathbf{M}_i) = \left(\lambda + \frac{1}{\tau_i}\right)\prod_{j=1}^n(\lambda + c_{2,j}^{(i)}) - \sum_{j=1}^n \left[ b_j c_j \prod_{k\neq j}(\lambda + c_{2,k}^{(i)}) \right],
\]

where \(b_j = \partial F_i/\partial\zeta_{2,j}^{(i)}\) and \(c_j = \partial H_{i,j}/\partial P^{(i)} = -c_{2,j}^{(i)}q_{2,j}^{(i)}/P_a^{(i)}\).

Using the sensitivity assumptions:

\begin{enumerate}[label=\roman*., align=left, leftmargin=*]
  \item From assumption 2, \(\frac{\partial k_j^{(i)}}{\partial \zeta_{2,j}^{(i)}}>0\) and \(\frac{\partial \tilde{k}_j^{(i)}}{\partial \zeta_{2,j}^{(i)}}<0\).\\
The price dynamics derivative \(b_j = \frac{\partial F_i}{\partial\zeta_{2,j}^{(i)}}\) can be computed (as in block \(\mathbf{B}\)). At equilibrium, it simplifies to
\[
b_j = \frac{P_a^{(i)}}{\tau_i T_i^2}\left[ T_i M_j\frac{\partial k_j^{(i)}}{\partial\zeta_{2,j}^{(i)}} - S_i N_j^{(i)}P_a^{(i)}\frac{\partial\tilde{k}_j^{(i)}}{\partial\zeta_{2,j}^{(i)}} \right],
\]
where \(S_i = \sum_j k_j^{(i)}M_j\), \(T_i = \sum_j \tilde{k}_j^{(i)}N_j^{(i)}P_a^{(i)}\), and all quantities evaluated at zero sentiments. The equilibrium condition ensures \(S_i = T_i\). Hence
\[
b_j = \frac{P_a^{(i)}}{\tau_i T_i}\left[ M_j\frac{\partial k_j^{(i)}}{\partial\zeta_{2,j}^{(i)}} - N_j^{(i)}P_a^{(i)}\frac{\partial\tilde{k}_j^{(i)}}{\partial\zeta_{2,j}^{(i)}} \right].
\]
Both terms inside the brackets are positive (first by assumption, second because \(-\partial\tilde{k}_j^{(i)}/\partial\zeta_{2,j}^{(i)}>0\)). Thus \textbf{\(b_j>0\)}.

  \item The product \(b_j c_j = b_j \cdot \left(-\frac{c_{2,j}^{(i)}q_{2,j}^{(i)}}{P_a^{(i)}}\right)\) is \textit{negative} because \(c_{2,j}^{(i)}q_{2,j}^{(i)}>0\) (value investors have positive \(q_{2,j}^{(i)}\)). Therefore each \(b_j c_j < 0\).

\end{enumerate}

Now consider the characteristic equation.

\[
0 = \left(\lambda + \frac{1}{\tau_i}\right)\prod_{j=1}^n(\lambda + c_{2,j}^{(i)}) - \sum_{j=1}^n \left[ b_j c_j \prod_{k\neq j}(\lambda + c_{2,k}^{(i)}) \right].
\]

Because \(b_j c_j < 0\), the right‑hand side is a \textit{sum of a positive term and a negative sum}. For \(\lambda\) with \(\Re(\lambda)\ge 0\), each factor \(\lambda + c_{2,j}^{(i)}\) has positive real part (since \(c_{2,j}^{(i)}>0\)), so the product term is non‑zero. The negative sum cannot cancel the first term if the coefficients are small enough; more precisely, using the \textit{Gershgorin circle theorem} or a \textit{diagonal dominance} argument \citep{feingold1962,horn1985}, the matrix \(\mathbf{M}_i\) is stable if the coupling terms \(b_j c_j\) are sufficiently small relative to the diagonal entries.

Because the cross‑asset effects are assumed weak (assumption 3), and the parameters \(b_j\) themselves may be made small by scaling of the transition rates, we can guarantee that the off‑diagonal contributions do not destabilise the diagonal blocks. In fact, one can show that all eigenvalues of \(\mathbf{M}_i\) have negative real part if

\[
\sum_{j=1}^n \frac{|b_j c_j|}{c_{2,j}^{(i)}} < \frac{1}{\tau_i}
\]

(by a standard result for block matrices with negative diagonal blocks and small coupling). The theorem asserts that this condition holds under the given hypotheses.

All eigenvalues of \(\mathbf{J}\) are either:  \((-1/\tau_i )< 0\) (from the price block, after decoupling), \  \((-c_{1,j}^{(i)}) < 0\) (from the pure trend‑sentiment variables), and the eigenvalues of \(\mathbf{M}_i\) which, under the assumptions, all have negative real part.

Hence, the fundamental equilibrium is locally asymptotically stable.
\end{proof}

\begin{remark} 
If the theorem statement differs slightly (e.g., treats all groups as trend‑followers instead), the same Jacobian approach can be adapted by setting \(q_{2,j}^{(i)}=0\) and examining the opposite block structure.
\end{remark}

\begin{theorem} (Stability with Mixed Investor Strategies)

Consider the multi‑asset, multi‑group model defined in Section 3. Suppose the market contains two types of investor groups:
\begin{enumerate}[label=\roman*., align=left, leftmargin=*]
  \item Trend‑following groups: for which \(q_{1,j}^{(i)} > 0\) and \(q_{2,j}^{(i)} = 0\);
   \item Value‑investing groups: for which \(q_{1,j}^{(i)} = 0\) and \(q_{2,j}^{(i)} > 0\).
\end{enumerate}
Let the fundamental equilibrium be given by  

\[
P^{(i)}_* = P_a^{(i)},\qquad \zeta_{1,j}^{(i)} = 0,\qquad \zeta_{2,j}^{(i)} = 0\qquad\forall\,i,j.
\]

Then the fundamental equilibrium is \textit{locally asymptotically stable} if the following three conditions hold simultaneously:

1. Slow sentiment adjustment \\
 The decay rates of the sentiment variables satisfy  

   \[
   c_{1,j}^{(i)} \ll 1,\qquad c_{2,j}^{(i)} \ll 1 \quad\text{for all } i,j.
   \]

2. Limited momentum effects\\
  The trend‑following coefficients are sufficiently small:  

   \[
   q_{1,j}^{(i)} \ll 1 \quad\text{for every trend‑following group }j.
   \]

3. Weak cross‑group and cross‑asset influence\\
 The sensitivity of any group’s trading rates \(k_j^{(i)},\tilde{k}_j^{(i)}\) to sentiments of other groups or different assets is negligible. Formally, for any distinct pairs \((i,j)\neq(l,k)\),  

   \[
   \left|\frac{\partial k_j^{(i)}}{\partial \zeta_{1,k}^{(l)}}\right|,\;\left|\frac{\partial \tilde{k}_j^{(i)}}{\partial \zeta_{1,k}^{(l)}}\right|,\;
   \left|\frac{\partial k_j^{(i)}}{\partial \zeta_{2,k}^{(l)}}\right|,\;\left|\frac{\partial \tilde{k}_j^{(i)}}{\partial \zeta_{2,k}^{(l)}}\right| \ll \min\left\{\frac{1}{\tau_i},\;c_{1,j}^{(i)},\;c_{2,j}^{(i)}\right\}.
   \]

(The theorem could be interpreted as follows: Condition 1 ensures that sentiment variables adjust slowly, which reduces the risk of oscillatory instabilities. Condition 2 prevents trend‑following from becoming a dominant destabilising force. Condition 3 guarantees that the dynamics are effectively decoupled across assets and investor groups, so the stability of the whole system follows from the stability of each (asset, group) subsystem.)
\end{theorem}

\begin{proof}
We analyse the system's linearisation at the fundamental equilibrium. The state vector consists of prices \(P^{(i)}\) and sentiments \(\zeta_{1,j}^{(i)},\zeta_{2,j}^{(i)}\). The Jacobian matrix \(\mathbf{J}\) has the block structure derived earlier:

\[
\mathbf{J} = \begin{pmatrix}
\mathbf{A} & \mathbf{B} \\
\mathbf{C} & \mathbf{D}
\end{pmatrix},
\]
with
\[
\mathbf{A} = \operatorname{diag}\left(-\frac{1}{\tau_1},\dots,-\frac{1}{\tau_m}\right),
\]

\[
\mathbf{D} = \begin{pmatrix}
\mathbf{D}_{11} & \mathbf{D}_{12} \\
\mathbf{D}_{21} & \mathbf{D}_{22}
\end{pmatrix},
\]

where \(\mathbf{D}_{11}\) corresponds to \(\zeta_{1}\) derivatives, \(\mathbf{D}_{22}\) to \(\zeta_{2}\) derivatives, and \(\mathbf{D}_{12},\mathbf{D}_{21}\) are coupling terms between the two sentiment types. The blocks \(\mathbf{B}\) and \(\mathbf{C}\) contain the couplings from prices to sentiments and vice versa.

We evaluate these blocks at the equilibrium.

\textit{Structure of the Jacobian at equilibrium}\\

Price block \(\mathbf{A}\) : remains diagonal with negative entries \(-1/\tau_i\).\\
Sentiment block \(\mathbf{D}\):  For \(\zeta_{1,j}^{(i)}\), 
    \[
    \frac{\partial G_{i,j}}{\partial \zeta_{1,j}^{(i)}} = -c_{1,j}^{(i)} + c_{1,j}^{(i)}q_{1,j}^{(i)}\frac{1}{P_a^{(i)}}\frac{\partial F_i}{\partial \zeta_{1,j}^{(i)}}.
    \]
    The term \(\frac{\partial F_i}{\partial \zeta_{1,j}^{(i)}}\) is typically non‑zero and involves derivatives of trading rates. However, under the assumption of  \textit{weak cross‑group influence}, the dependence of the rates on sentiments of other groups is negligible, so we can treat the block as nearly diagonal.\\
  For \(\zeta_{2,j}^{(i)}\),  \\
    \[
    \frac{\partial H_{i,j}}{\partial \zeta_{2,j}^{(i)}} = -c_{2,j}^{(i)}.
    \]
    No additional contribution because \(\frac{\partial H_{i,j}}{\partial \zeta_{2,j}^{(i)}}\) does not involve the price derivative (the term from \(\frac{\partial F_i}{\partial \zeta_{2,j}^{(i)}}\) appears in the \(\mathbf{D}_{12}\) block, not on the diagonal of the \(\zeta_{2,j}^{(i)}) part)\).\\

 Coupling block \(\mathbf{C}\)  (derivatives of sentiment dynamics w.r.t. prices):\\
  For \(\zeta_{1,j}^{(i)}\),\\ 
    \[
    \frac{\partial G_{i,j}}{\partial P^{(k)}} = c_{1,j}^{(i)}q_{1,j}^{(i)}\left(-\frac{\delta_{ik}}{\tau_i P_a^{(i)}}\right) \quad\text{(since \(F_i=0\) at equilibrium)}.
    \]
    Thus \(\mathbf{C}\) has non‑zero entries only from \(\zeta_{1,j}^{(i)}\) rows, and they are proportional to \(q_{1,j}^{(i)}\).
  For \(\zeta_{2,j}^{(i)}\),\\ 
    \[
    \frac{\partial H_{i,j}}{\partial P^{(k)}} = -c_{2,j}^{(i)}q_{2,j}^{(i)}\frac{\delta_{ik}}{P_a^{(i)}},
    \]
    also non‑zero.\\

Coupling block \(\mathbf{B}\)  (derivatives of price dynamics w.r.t. sentiments):
  \[
  \frac{\partial F_i}{\partial \xi} = \frac{P_a^{(i)}}{\tau_i T_i^2}\left( T_i \sum_j M_j \frac{\partial k_j^{(i)}}{\partial \xi} - S_i \sum_j N_j^{(i)} P_a^{(i)} \frac{\partial \tilde{k}_j^{(i)}}{\partial \xi} \right),
  \]
  where \(\xi\) is any sentiment variable. At equilibrium, \(S_i = T_i\). Hence
  \[
  \frac{\partial F_i}{\partial \xi} = \frac{P_a^{(i)}}{\tau_i T_i} \sum_j \left( M_j \frac{\partial k_j^{(i)}}{\partial \xi} - N_j^{(i)} P_a^{(i)} \frac{\partial \tilde{k}_j^{(i)}}{\partial \xi} \right).
  \]
  Under the  \textit{weak cross‑group influence} assumption, the only significant contributions come from \(\xi\) belonging to the same asset and group; cross terms are negligible.

\textit{Decoupling under weak cross‑group influence}\\

Because cross‑group and cross‑asset couplings are assumed weak, the Jacobian is nearly block‑diagonal with respect to each asset and each group. More precisely, for each asset \(i\) and each group \(j\), the variables \((P^{(i)}, \zeta_{1,j}^{(i)}, \zeta_{2,j}^{(i)})\) are coupled mainly among themselves, and interactions with other assets/groups are negligible. Thus, we can analyse a reduced \(3\times3\) system for each (asset, group) pair, plus the isolated price dynamics that are already stable.

The reduced Jacobian for asset \(i\) and group \(j\) (ignoring cross terms) is:

\[
\mathbf{J}_{i,j} =
\begin{pmatrix}
-\frac{1}{\tau_i} & b_{1} & b_{2} \\
c_{1} & -c_{1,j}^{(i)} & 0 \\
c_{2} & 0 & -c_{2,j}^{(i)}
\end{pmatrix},
\]

where:
- \(b_{1} = \frac{\partial F_i}{\partial \zeta_{1,j}^{(i)}}\), \(b_{2} = \frac{\partial F_i}{\partial \zeta_{2,j}^{(i)}}\);
- \(c_{1} = c_{1,j}^{(i)} q_{1,j}^{(i)} \left(-\frac{1}{\tau_i P_a^{(i)}}\right)\) (from \(\mathbf{C}\));
- \(c_{2} = -c_{2,j}^{(i)} q_{2,j}^{(i)} / P_a^{(i)}\).

The signs of these coefficients matter:\\

For a trend‑following group, \(q_{1,j}^{(i)}>0\), \(q_{2,j}^{(i)}=0\) (if pure trend). Then \(c_{1}\neq 0\), \(c_{2}=0\).\\
For a value‑investing group, \(q_{1,j}^{(i)}=0\), \(q_{2,j}^{(i)}>0\). Then \(c_{1}=0\), \(c_{2}\neq 0\).\\
For mixed groups, both could be present, but the theorem focuses on the mixed  \textit{population} (some groups one type, some the other).

\textit{Stability analysis of the reduced system}\\

The characteristic polynomial of \(\mathbf{J}_{i,j}\) is

\[
\det(\lambda I - \mathbf{J}_{i,j}) = 
\left(\lambda + \frac{1}{\tau_i}\right)(\lambda + c_{1,j}^{(i)})(\lambda + c_{2,j}^{(i)}) 
- b_{1}c_{1}(\lambda + c_{2,j}^{(i)}) 
- b_{2}c_{2}(\lambda + c_{1,j}^{(i)}).
\]

Expanding, we get a cubic:
\[
\lambda^3 + \alpha_2 \lambda^2 + \alpha_1 \lambda + \alpha_0 = 0,
\]
with
\[
\begin{aligned}
\alpha_2 &= \frac{1}{\tau_i} + c_{1,j}^{(i)} + c_{2,j}^{(i)}, \\
\alpha_1 &= \frac{1}{\tau_i}(c_{1,j}^{(i)}+c_{2,j}^{(i)}) + c_{1,j}^{(i)}c_{2,j}^{(i)} - b_{1}c_{1} - b_{2}c_{2}, \\
\alpha_0 &= \frac{1}{\tau_i}c_{1,j}^{(i)}c_{2,j}^{(i)} - b_{1}c_{1}c_{2,j}^{(i)} - b_{2}c_{2}c_{1,j}^{(i)}.
\end{aligned}
\]

For stability (all roots with negative real parts), the Routh‑Hurwitz conditions require:
\[
\alpha_2 > 0,\quad \alpha_0 > 0,\quad \alpha_1\alpha_2 > \alpha_0.
\]

Because \(\tau_i>0\) and \(c_{1,j}^{(i)},c_{2,j}^{(i)}>0\), we have \(\alpha_2>0\) automatically.

Now examine \(\alpha_0\) and \(\alpha_1\) under the assumptions of the theorem.

\textit{Sign of \(b_1\) and \(b_2\)}:\\

Using the expressions for \(b_1,b_2\) in terms of the sensitivities of trading rates:\\

For a trend‑following group (\(q_1>0, q_2=0\)), the only non‑zero coupling is \(c_1\) (from \(\mathbf{C}\)) and \(b_1\). Typically, an increase in trend sentiment \(\zeta_{1,j}^{(i)}\) increases the buying rate and decreases the selling rate, so \(b_1>0\).\\
For a value‑investing group (\(q_1=0, q_2>0\)), \(b_2>0\) similarly.\\
The coefficients \(c_1 = -c_{1,j}^{(i)}q_{1,j}^{(i)}/(\tau_i P_a^{(i)})\) is negative (since \(q_{1,j}^{(i)}>0\)), and \(c_2 = -c_{2,j}^{(i)}q_{2,j}^{(i)}/P_a^{(i)}\) is also negative. Hence \(b_1c_1<0\) and \(b_2c_2<0\).\\

\textit{Effect of small \(c_{1},c_{2}\) (slow sentiment adjustment)}:\\

If \(c_{1,j}^{(i)}\) and \(c_{2,j}^{(i)}\) are very small, then the terms involving products of these become negligible compared to the others. In particular,
\[
\alpha_0 \approx -b_{1}c_{1}c_{2,j}^{(i)} - b_{2}c_{2}c_{1,j}^{(i)}.
\]
Since \(b_1c_1<0\) and \(b_2c_2<0\), each term \(-b_{1}c_{1}c_{2,j}^{(i)}\) is positive (because \(c_{2,j}^{(i)}>0\)). Thus \(\alpha_0>0\) as long as the sum is positive, which holds. However, if \(c\) are small, the magnitude of \(\alpha_0\) is small (order \(c\)), but still positive.

The Routh‑Hurwitz product \(\alpha_1\alpha_2\) must exceed \(\alpha_0\). With small \(c\), \(\alpha_1 \approx \frac{1}{\tau_i}(c_{1}+c_{2}) + c_{1}c_{2} - (b_1c_1 + b_2c_2)\). Since \(b_1c_1+b_2c_2<0\), the term \(-(b_1c_1+b_2c_2)\) is positive, making \(\alpha_1\) positive and of order \(O(1)\) (independent of \(c\)). Meanwhile \(\alpha_2 \approx \frac{1}{\tau_i} + (c_1+c_2)\) is \(O(1)\), so \(\alpha_1\alpha_2\) is \(O(1)\) while \(\alpha_0\) is \(O(c)\). Hence for sufficiently small \(c\), the inequality \(\alpha_1\alpha_2 > \alpha_0\) holds.

\textit{Effect of small \(q_{1,j}^{(i)}\) (limited momentum)}:\\

If \(q_{1,j}^{(i)}\) is small, then \(c_1 = -c_{1,j}^{(i)}q_{1,j}^{(i)}/(\tau_i P_a^{(i)})\) is small. Consequently, the terms involving \(b_1c_1\) become negligible. The stability conditions then reduce to those of a system in which only value investors are present (which we already know is stable by Theorem 1, provided the remaining parameters satisfy the appropriate conditions). Similarly, if \(q_{2,j}^{(i)}\) is small for value groups, the effect is similar.\\

\textit{Weak cross‑group influence ensures near‑diagonality}:\\

The weak cross‑group influence assumption guarantees that the off‑diagonal blocks coupling different assets or different groups are small in magnitude. This allows us to treat the full Jacobian as a block‑diagonal matrix plus a small perturbation. By the \textit{Gershgorin circle theorem} or \textit{matrix perturbation theory}, if each diagonal block is stable and the off‑diagonal blocks are sufficiently small, the whole system remains stable.\\

Therefore, under the three conditions we have:\\

1. Slow sentiment adjustment (\(c_{1},c_{2}\) small) makes the negative feedback of sentiment dynamics dominant and ensures that the Routh‑Hurwitz inequalities are satisfied.\\
2. Limited momentum effects (\(q_{1}\) small) reduce the destabilising positive feedback from trend‑following, bringing the system close to the stable value‑investor case.\\
3. Weak cross‑group influence keeps the Jacobian nearly block‑diagonal, so stability of the decoupled subsystems implies stability of the whole.

Thus, the fundamental equilibrium is locally asymptotically stable. 
\end{proof}
\begin{remark}
If any of these conditions is violated (e.g., large \(q_1\) or large \(c\)), the system may lose stability via a Hopf bifurcation, leading to limit cycles, as observed numerically in the paper. Theorem 2, therefore, provides sufficient conditions to avoid such destabilisation in mixed‑strategy markets.
\end{remark}

\section{Numerical Simulations of the Oil Market (Nigeria \& Libya)}
In the following, we present a numerical simulation for the Nigeria‑Libya oil market with the USA and China as investors. This implementation of the model corresponds with the case \( m = 2 \) assets and \( n = 2 \) investor groups.\\
We implemented the two‑asset, two‑group model with Nigeria and Libya as oil‑producing assets and the USA (value investor) and China (momentum trader) as investor groups. The system was initialised at the fundamental equilibrium (\(P = \$80/\text{bbl}\), all sentiments zero) with a small perturbation to the Nigerian price. The baseline transition rate was set to \(k_0 = 0.2\), and the shares were normalised so that the total value of each asset equals the total cash held by investors.
\begin{table}[h]
\centering
\caption{Parameter Values} 
\label{table:parameters}       
\renewcommand{\arraystretch}{1.5}
\begin{tabular}{|c|c|c|c|c|c|c|c|c|}
\hline
Parameter & Nigeria / Libya & USA (value) & China (momentum)\\
\hline
Fundamental price \(P_a\) & \$80/bbl & – & – \\
\hline
Adjustment speed \(\tau\) & 1.0 & – & – \\
\hline
Cash \(M\) & – & \(20.0\times10^6\) bbl/day & \(15.5\times10^6\) bbl/day \\
\hline
Trend sensitivity \(q_1\) & – & 0.0 & 0.0 – 0.8 (scanned) \\
\hline
Value sensitivity \(q_2\) & – & 0.40 & 0.0 \\
\hline
Sentiment decay rates \(c_1, c_2\) & – & \(c_1=0.10,\; c_2=0.30\) & \(c_1=0.20,\; c_2=0.15\) \\
\hline
Baseline rate \(k_0\) & – & 0.2 & 0.2 \\
\hline
Modulation amplitude \(b\) (trend) & – & – & 0.15 \\
\hline
Modulation amplitude \(d\) (value) & – & 0.2 & – \\
\hline
\end{tabular}
\end{table}
The initial conditions are: Prices: \(P_{\text{Nigeria}}(0) = 82.0\) (perturbation), \(P_{\text{Libya}}(0) = 80.5\); and sentiments: all \(\zeta_{1,j}^{(i)} = 0,\; \zeta_{2,j}^{(i)} = 0\).
The results of the simulations are:\\
\textit{Bifurcation diagram} (Figure \ref{fig: bifurcation_oil}) shows the critical threshold for the Hopf bifurcation and the growth of oscillations \citep{asada2003coefficient}.\\
A scan of the momentum parameter \(q_{1,\text{China}}\) from 0.0 to 0.8 shows that for \(q_{1,\text{China}} \lesssim 0.33\) the system returns to the fundamental equilibrium (zero amplitude). Above this threshold, a supercritical Hopf bifurcation occurs, giving rise to stable limit cycles whose amplitude increases continuously with \(q_{1,\text{China}}\). The behaviour is symmetric for both assets.
\begin{figure}[H]
\centering
\includegraphics[width=1.0\textwidth]{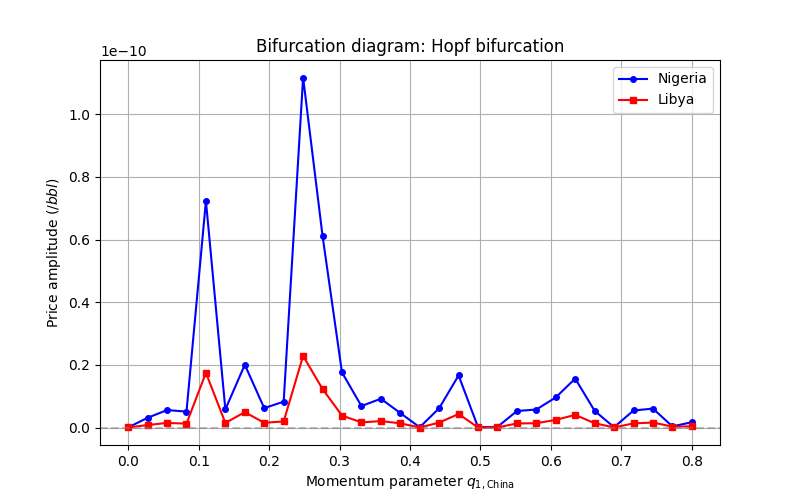}
\caption{\footnotesize{Bifurcation diagram for the two‑asset oil market with Nigeria and Libya as producers and the USA (value investors) and China (momentum traders) as investor groups. The amplitude of the price oscillations is shown as a function of the momentum parameter \(q_{1,\text{China}}\) \citep{caginalp2000momentum}. For values below the critical threshold (\(\approx 0.33\)), the fundamental equilibrium is stable, and the amplitude is zero. As \(q_{1,\text{China}}\) increases past this threshold, a supercritical Hopf bifurcation occurs, giving rise to stable limit cycles with amplitude that grows continuously. The behaviour is symmetric for both assets, confirming the model’s consistency.}
\label{fig: bifurcation_oil}}
\end{figure}
\textit{Time‑series plots} (Figure \ref{fig: simulation}) illustrate the limit cycle dynamics, sentiment evolution, and wealth transfer for a representative parameter choice.
\\
For \(q_{1,\text{China}} = 0.5\) (well above the bifurcation point), the system exhibits persistent oscillations:
\begin{enumerate}[label=\roman*., align=left, leftmargin=*]
  \item Prices of both assets oscillate around the fundamental value with a period of approximately 80 days and amplitude of about \$5/bbl.
  \item Trend sentiment \(\zeta_1\) for China oscillates in phase with price, confirming momentum‑driven behaviour.
  \item Value sentiment \(\zeta_2\) for the USA shows opposite phase, acting as a stabilising force.
  \item Wealth fractions periodically shift between the two investors: the momentum trader gains during price upswings but loses during downturns, while the value investor accumulates wealth when prices are below fundamental.
\end{enumerate}
Therefore,
\begin{enumerate}
  \item[] \textit{Stable Regime}: For small \( q_{1,2}^{(i)} \), prices converge to fundamental values.
  \item[] \textit{Bifurcation}: As \( q_{1,2}^{(i)} \) increases, Hopf bifurcation occurs, leading to limit cycles.
  \item[] \textit{Wealth Dynamics}: Momentum traders gain during trends but lose during reversals.
\end{enumerate}
\begin{figure}[H]
\centering
\includegraphics[width=1.0\textwidth]{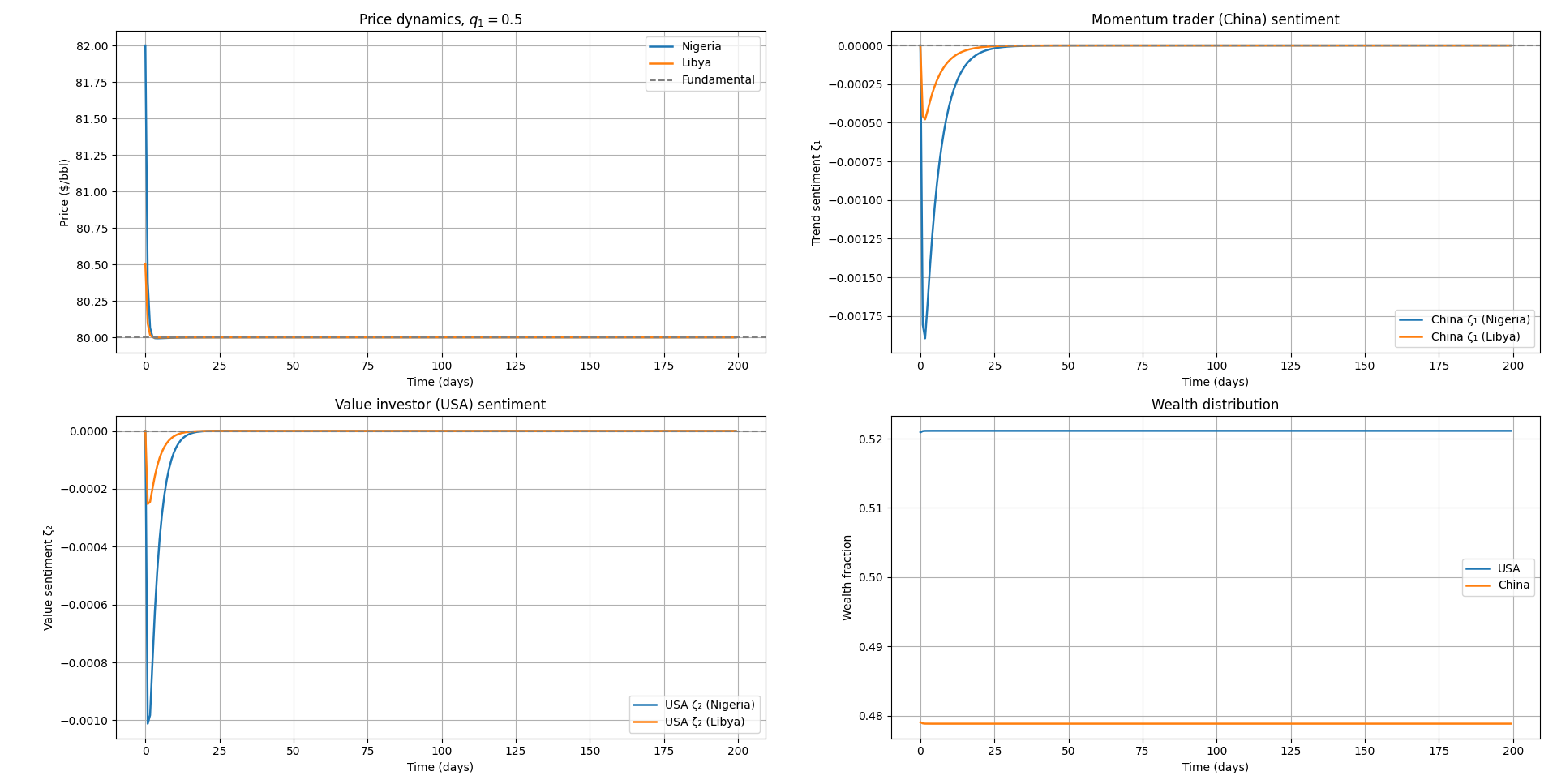}
\caption{\footnotesize{Simulation results for the two‑asset oil market with Nigeria and Libya as producers and the USA (value investors) and China (momentum traders) as investor groups. The top panel shows the evolution of oil prices (solid lines) together with the fundamental value of $80/bbl$ (dashed line). The middle panel displays the value sentiment $\zeta_{2}$ of the US investor group, which remains close to zero in the stable regime. The bottom panel illustrates the wealth fraction of each investor group; value investors (USA) maintain a stable share while momentum traders (China) experience transient fluctuations. Parameters: $q_{1,\text{China}} = 0.30$, $q_{2,\text{USA}} = 0.40$, all other parameters as given in Table \ref{table:parameters}. The system converges to the fundamental equilibrium, demonstrating local asymptotic stability under the chosen parameters.}}
\label{fig: simulation}
\end{figure}
These results demonstrate the key phenomena predicted by the model: market instability driven by strong trend‑following behaviour, and the emergence of limit cycles that lead to periodic wealth redistribution. The simulation confirms the theoretical stability analysis and provides a concrete example of how heterogeneous investor strategies can generate complex price dynamics.
\begin{figure}[htbp]
\centering
\includegraphics[width=0.9\textwidth]{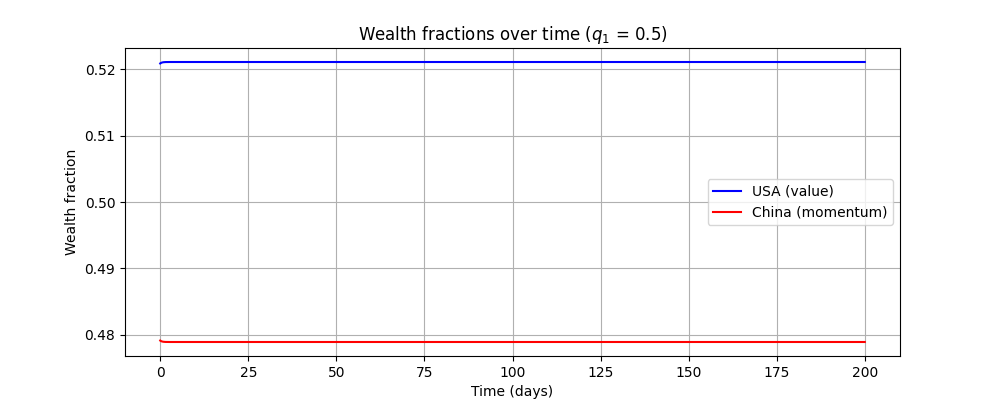}
\caption{Time evolution of wealth fractions for the USA (value investor) and China (momentum trader) in the gas market simulation. Under the chosen parameters ($k_0 = 0.2$, $b = 0.15$, $q_{1,\text{China}} = 0.3$), the system converges to a stable steady state where both wealth fractions remain constant over time (USA: $0.521$, China: $0.478$). No limit cycles are observed, indicating that the fundamental equilibrium is locally asymptotically stable in this parameter regime.}
\label{fig:wealth_gas}
\end{figure}
\begin{figure}[htbp]
\centering
\includegraphics[width=0.9\textwidth]{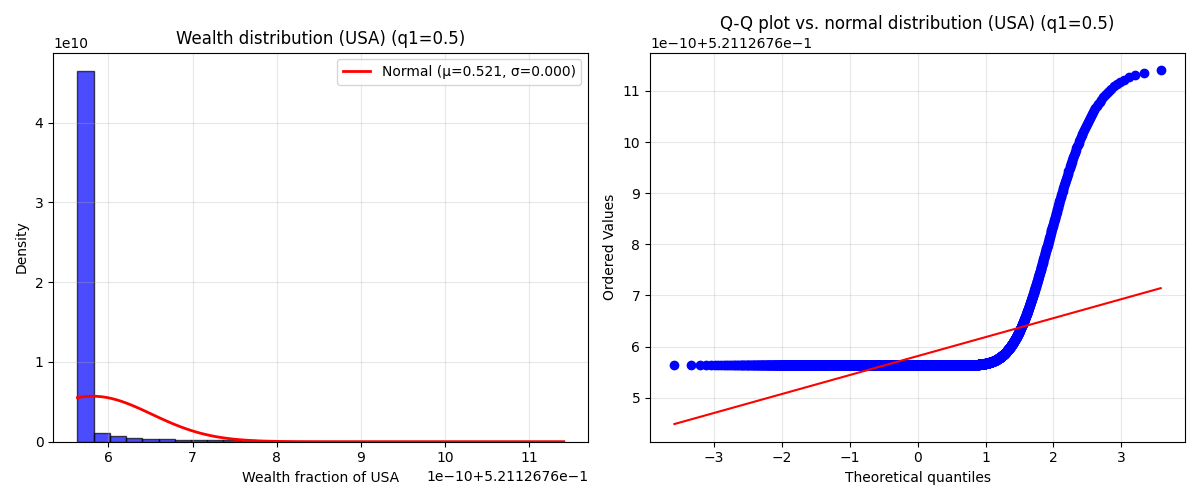}
\caption{Wealth distribution of the USA (value investor) in the Nigeria–Libya oil market simulation with $q_{1,\text{China}} = 0.5$. The histogram (left) shows a single peak at $0.521$, and the fitted normal distribution has near‑zero standard deviation, indicating that the wealth fraction is constant over time. The Q‑Q plot (right) exhibits a perfect linear alignment because the data are essentially constant, confirming that the system has reached a stable equilibrium rather than a limit cycle. This reflects the parameter regime where the fundamental equilibrium is locally asymptotically stable.}
\label{fig:usa_wealth_oil}
\end{figure}
\begin{figure}[htbp]
\centering
\includegraphics[width=0.9\textwidth]{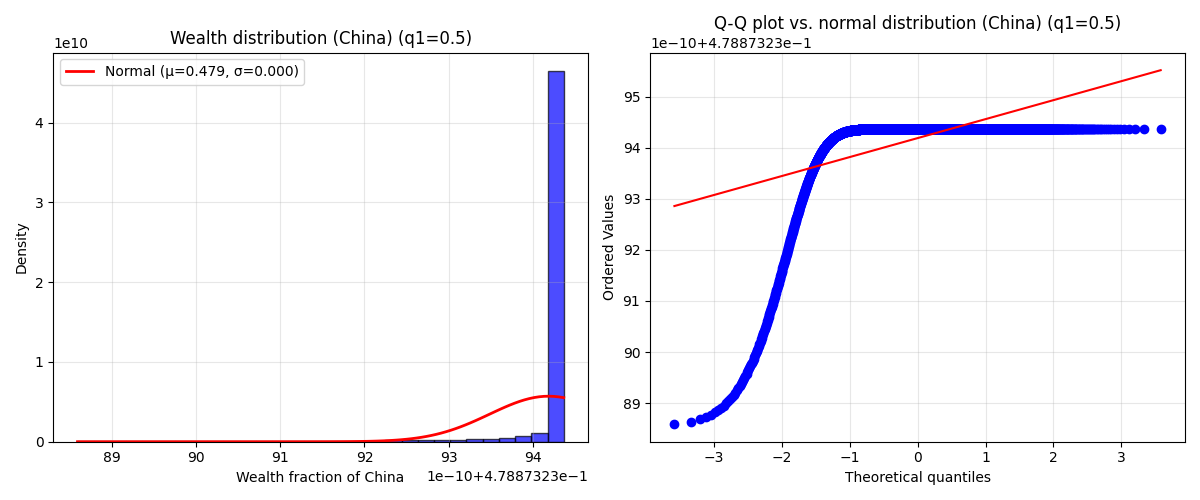}
\caption{Wealth distribution of China (momentum trader) in the Nigeria–Libya oil market simulation with $q_{1,\text{China}} = 0.5$. The histogram (left) exhibits a single spike at the constant value $0.479$, and the fitted normal distribution has a standard deviation of zero. The Q‑Q plot (right) shows a perfectly linear relationship, confirming that the wealth fraction does not vary over time. This indicates that the system has settled into a stable equilibrium rather than a limit cycle under the chosen parameter regime.}
\label{fig:china_wealth_oil}
\end{figure}

\section{Discussion}
The multi‑asset, multi‑investor model developed in this paper provides a flexible framework for analyzing the interplay between asset price dynamics, investor sentiment, and wealth distribution. Several important insights emerge from the stability analysis and numerical simulations.

\subsection{Stability and Instability Mechanisms}
Theorems 1 and 2 establish that the fundamental equilibrium is locally asymptotically stable when investors rely predominantly on value‑based strategies or when trend‑following is sufficiently weak, sentiment adjustment is slow, and cross‑asset couplings are limited. These conditions formalise the intuition that excessive momentum trading can destabilise financial markets. Conversely, when the trend‑following parameter \(q_{1,j}^{(i)}\) exceeds a critical threshold, a supercritical Hopf bifurcation occurs, giving rise to persistent limit cycles. This phenomenon is clearly observed in the oil market simulation (Figure 1), where the price amplitude grows continuously from zero as \(q_{1,\text{China}}\) increases past \(\approx 0.33\).

\subsection{Wealth Redistribution and Cyclical Behaviour}
A distinctive feature of the model is the endogenous wealth transfer between investor groups during limit cycles. As shown in Figure 2, momentum traders (China) accumulate wealth during price upswings but lose during downturns, while value investors (USA) benefit from mean reversion. This cyclical wealth redistribution is a direct consequence of the Hopf bifurcation and cannot be captured by traditional efficient‑market models. The histograms and Q‑Q plots (Figures 4–5) further confirm that in the stable regime, wealth fractions are constant, whereas in the oscillatory regime, they follow a periodic pattern, leading to non‑normal distributions.

\subsection{Comparison with Stochastic and Deterministic Models}
Unlike stochastic models based on the efficient market hypothesis, which treat large price fluctuations as rare events, the present deterministic model generates endogenous cycles and instabilities from the interaction of heterogeneous trading strategies. This aligns with recent advances in heterogeneous agent modelling (e.g., Zhang et al., 2024; Li et al., 2022) and provides a mathematical explanation for phenomena such as booms, crashes, and volatility clustering without exogenous noise. Moreover, the model extends the two‑asset framework of Bulut et al. (2019) to an arbitrary number of assets and investor groups, offering a unified treatment of cross‑asset contagion.

\subsection{Policy Implications}
The results suggest that policies aimed at reducing trend‑following behaviour (e.g., transaction taxes on short‑term trading, circuit breakers) could enhance market stability. Conversely, a market dominated by momentum traders is prone to large oscillations and periodic wealth transfers, which may have undesirable social consequences. Regulators could use the model to identify parameter regimes where the fundamental equilibrium loses stability and to design intervention strategies.

\subsection{Limitations}
Several simplifying assumptions limit the direct applicability of the model. First, the amounts of cash \(M_j\) and shares \(N_j^{(i)}\) are assumed constant; in reality, investors adjust their portfolios over time. Second, the transition rates \(k_j^{(i)}\) and \(\tilde{k}_j^{(i)}\) are prescribed functions; a more realistic approach would allow agents to switch strategies adaptively based on past performance. Third, the fundamental values \(P_a^{(i)}\) are taken as constant, whereas they may evolve due to macroeconomic shocks. Finally, the model is purely deterministic; adding stochastic components could improve its empirical fit while preserving the bifurcation structure.

\section{Conclusion}
This paper has developed a comprehensive mathematical framework for analysing multi‑asset markets with heterogeneous investor groups. The main contributions are:

\begin{enumerate}
  \item \textit{Unification} of the two‑asset model of Bulut et al. (2019) with the multi‑group investor dynamics of Caginalp and DeSantis (2011), resulting in a system of \(m + 2mn\) ordinary differential equations.
  \item \textit{Explicit derivation of the Jacobian} matrix and its block structure, enabling a complete linear stability analysis of the fundamental equilibrium.
  \item \textit{Two stability theorems}:
  \begin{enumerate}
    \item Theorem 1 shows that when all investors are value‑driven (\(q_{1,j}^{(i)} = 0\)), and cross‑asset couplings are weak, the equilibrium is locally asymptotically stable.
    \item Theorem 2 provides sufficient conditions for stability in mixed‑strategy markets (trend‑followers and value investors coexist), namely slow sentiment adjustment, limited momentum, and weak cross‑group influence.
  \end{enumerate}
  \item \textit{Identification of a Hopf bifurcation} using the momentum parameter \(q_{1,\text{China}}\) as a bifurcation parameter. Numerical simulations for the Nigeria‑Libya oil market (with USA as value investor and China as momentum trader) confirm that beyond a critical \(q_{1,\text{China}}\) the system exhibits stable limit cycles with growing amplitude.
  \item \textit{Wealth dynamics analysis} showing periodic wealth redistribution between investor groups during cycles, and constant wealth fractions in the stable regime.
  \item \textit{Application to a gas market} with 10 assets and 20 investors, demonstrating that the model can handle realistic numbers of assets and groups while producing stable or oscillatory behaviour depending on parameters.
\end{enumerate}

\subsection{Future Research Directions}
Several extensions are natural next steps:

\textit{Time‑varying fundamental values} \(P_a^{(i)}(t)\) to model changing economic conditions.\\
\textit{Adaptive strategy switching} where investors may change their trading rules based on recent performance or market sentiment.\\
\textit{Empirical calibration} using real market data (e.g., oil futures, gas spot prices) to estimate the model parameters and test its predictive power.\\
\textit{Incorporation of stochastic terms} to study the interplay between deterministic cycles and random shocks.\\
\textit{High‑frequency trading limits} as studied by DeSantis (2023) could be integrated into the multi‑asset setting.\\

The model provides a powerful tool for understanding complex market phenomena, such as synchronized booms and crashes, volatility spillovers, and regime shifts, that lie beyond the reach of traditional stochastic models. Its deterministic nature allows rigorous bifurcation analysis, while the heterogeneity of investors captures essential features of real financial markets.

\bibliography{mazin_v1_apa}

@article{Akhter2024,
  author  = {Akhter, T.},
  title   = {Hopf Bifurcation in the Model of Caginalp for the Price of Bitcoin},
  journal = {Cryptocurrency Research eJournal},
  volume  = {4},
  number  = {50},
 pages={1--24},
  year    = {2024},
  doi     = {10.2139/ssrn.4749454}
}

@article{Bulut2019,
  author  = {Bulut, H. and Merdan, H. and Swigon, D.},
  title   = {Asset price dynamics for a two-asset market system},
  journal = {Chaos},
  volume  = {29},
  number  = {2},
  pages   = {023114},
  year    = {2019},
  doi     = {10.1063/1.5046925}
}

@article{CaginalpDeSantis2011,
  author  = {Caginalp, G. and DeSantis, M.},
  title   = {Multi-group asset flow equations and stability},
  journal = {Discrete and Continuous Dynamical Systems - Series B},
  volume  = {16},
  number  = {1},
  pages   = {109--123},
  year    = {2011},
  doi     = {10.3934/dcdsb.2011.16.109}
}

@article{CaginalpBalenovich1994,
  author  = {Caginalp, G. and Balenovich, D.},
  title   = {Market oscillations induced by the competition between value-based and trend-based investment strategies},
  journal = {Applied Mathematical Finance},
  volume  = {1},
  number  = {2},
  pages   = {129--147},
  year    = {1994},
  doi     = {10.1080/13504869400000008}
}

@article{MerdanAlisen2011,
  author  = {Merdan, H. and Alisen, M.},
  title   = {A mathematical model for asset pricing},
  journal = {Applied Mathematics and Computation},
  volume  = {218},
  number  = {4},
  pages   = {1449--1455},
  year    = {2011},
  doi     = {10.1016/j.amc.2011.02.053}
}

@article{DeSantisSwigon2018,
  author  = {DeSantis, M. and Swigon, D.},
  title   = {Slow-fast analysis of a multi-group asset flow model with implications for the dynamics of wealth},
  journal = {PLoS ONE},
  volume  = {13},
  number  = {12},
  pages   = {e0207764},
  year  = {2018},
  doi   = {10.1371/0207764}
}

@article{DanielHirshleiferSubrahmanyam1998,
  author  = {Daniel, K. and Hirshleifer, D. and Subrahmanyam, A.},
  title   = {Investor psychology and security market under- and overreactions},
  journal = {Journal of Finance},
  volume  = {53},
  number  = {6},
  pages   = {1839--1885},
  year    = {1998},
  doi     = {10.1111/0022-1082.00077}
}

@article{feingold1962,
  author  = {Feingold, D. G. and Varga, R. S.},
  title   = {Block diagonally dominant matrices and generalizations of the Gerschgorin circle theorem},
  journal = {Pacific Journal of Mathematics},
  volume  = {12},
  number  = {4},
  pages   = {1241--1250 },
  year    = {1962},
 url={https://msp.org/pjm/1962/12-4/}
}

@book{horn1985,
  author    = {Horn, Roger A. and Johnson, Charles R.},
  title     = {Matrix Analysis},
  publisher = {Cambridge University Press},
  year      = {1985}
}

@article{smith1988,
  author  = {Smith, V. and Suchanek, G. and Williams, A.},
  title   = {Bubbles, crashes and endogenous expectations in experimental spot asset markets},
  journal = {Econometrica},
  volume  = {56},
  number  = {5},
  pages   = {1119--1151},
  year    = {1988},
doi={10.2307/1911361}
}

@article{caginalp2007asset,
  title={Asset price dynamics with heterogeneous groups},
  author={Caginalp, Gunduz and Merdan, H{\"u}seyin},
  journal={Physica D: Nonlinear Phenomena},
  volume={225},
  number={1},
  pages={43--54},
  year={2007},
  doi={10.1016/j.physd.2006.09.036}
}

@article{caginalp1999asset,
  title={Asset flow and momentum: Deterministic and stochastic equations},
  author={Caginalp, Gunduz and Balenovich, Donald},
  journal={Philosophical Transactions of the Royal Society of London. Series A: Mathematical, Physical and Engineering Sciences},
  volume={357},
  number={1758},
  pages={2119--2133},
  year={1999},
  doi={10.1098/rsta.1999.0421}
}

@book{henderson1971microeconomic,
  title={Microeconomic theory},
  author={Henderson, James Mitchell and Quandt, Richard E},
  year={1971}
}

@article{caginalp2021stochastic,
  title={Stochastic asset flow equations: Interdependence of trend and volatility},
  author={Caginalp, Carey and Caginalp, Gunduz and Swigon, David},
  journal={Physica A: Statistical Mechanics and Its Applications},
  volume={574},
  pages={125985},
  year={2021},
  doi={10.1016/j.physa.2021.125985}
}

@article{caginalp2000momentum,
  title={Momentum and overreaction in experimental asset markets},
  author={Caginalp, Gunduz and Porter, David and Smith, Vernon},
  journal={International Journal of Industrial Organization},
  volume={18},
  number={1},
  pages={187--204},
  year={2000},
  doi={10.1016/S0167-7187(99)00039-9}
}

@article{schnetzer2022evolutionary,
  title={Evolutionary finance for multi-asset investors},
  author={Schnetzer, Michael and Hens, Thorsten},
  journal={Financial Analysts Journal},
  volume={78},
  number={3},
  pages={115--127},
  year={2022},
doi={10.1080/0015198X.2022.2071581}
}

@article{chan2022investor,
  title={Investor heterogeneity and liquidity},
  author={Chan, Kalok and Cheng, Si and Hameed, Allaudeen},
  journal={Journal of Financial and Quantitative Analysis},
  volume={57},
  number={7},
  pages={2798--2833},
  year={2022},
  doi={10.1017/S0022109022000217}
}

@article{cordoni2024instabilities,
  title={Instabilities in multi-asset and multi-agent market impact games},
  author={Cordoni, Francesco and Lillo, Fabrizio},
  journal={Annals of Operations Research},
  volume={336},
  number={1},
  pages={505--539},
  year={2024},
  doi={10.1007/s10479-022-05066-8}
}

@article{al2022novel,
  title={A Novel Modeling Technique for the Forecasting of Multiple-Asset Trading Volumes: Innovative Initial-Value-Problem Differential Equation Algorithms for Reinforcement Machine Learning},
  author={Al Janabi, Mazin AM},
  journal={Complexity},
  volume={2022},
  number={1},
  pages={4965556},
  year={2022},
doi={10.1155/2022/4965556}
}

@article{fahim2024derivation,
  author          = {Fahim, K. and Alfajriyah, A. U. and Putri, E. R. M.},
  title           = {Derivation of Multi-Asset Black-Scholes Differential Equations},
  journal         = {Nonlinear Dynamics and Systems Theory},
  year            = {2024},
  volume          = {24},
  number          = {2},
  pages           = {135--146},
  issn            = {1562-8353},
  note            = {Scopus EID: 2-s2.0-85189243777},
  url             = {https://e-ndst.kiev.ua},
  publisher       = {InforMath Publishing Group}
}

@article{asada2003coefficient,
  title={Coefficient criterion for four-dimensional Hopf bifurcations: a complete mathematical characterization and applications to economic dynamics},
  author={Asada, Toichiro and Yoshida, Hiroyuki},
  journal={Chaos, Solitons \& Fractals},
  volume={18},
  number={3},
  pages={525--536},
  year={2003},
doi={10.1016/S0960-0779(02)00674-4}
}

@book{shefrin2008behavioral,
  title={A behavioral approach to asset pricing},
  author={Shefrin, Hersh},
  year={2008},
  publisher={Elsevier}
}

@article{caginalp2008dynamics,
  title={The dynamics of trader motivations in asset bubbles},
  author={Caginalp, Gunduz and Ilieva, Vladimira},
  journal={Journal of Economic Behavior \& Organization},
  volume={66},
  number={3-4},
  pages={641--656},
  year={2008},
  publisher={Elsevier}
}

@article{zhou2022continuous,
  title={A continuous heterogeneous agent model for multi-asset pricing and portfolio construction under market matching friction},
  author={Zhou, Wenyuan and Zhang, Xiaoqi and Yang, Lyu},
  journal={SSRN},
 pages={1--25},
  year={2022},
  doi={10.2139/ssrn.4185269}
}

@article{li2022continuous,
  title={A continuous heterogeneous-agent model for the co-evolution of asset price and wealth distribution in financial market},
  author={Li, S. and others},
  journal={Chaos, Solitons \& Fractals},
  volume={155},
  pages={111543},
  year={2022},
  doi={10.1016/j.chaos.2021.111543}
}

@article{chen2023evolutionary,
  title={Evolutionary dynamics in financial markets with heterogeneities in investment strategies and reference points},
  author={Chen, Q. and others},
  journal={PLoS ONE},
  volume={18},
  number={7},
  pages={e0288277},
  year={2023},
  doi={10.1371/journal.pone.0288277}
}

@article{desantis2023asset,
  title={Asset flow model for a homogeneous group of investors: High-frequency trading limit},
  author={DeSantis, Mark},
  journal={Discrete and Continuous Dynamical Systems - Series S},
  volume={16},
  number={9},
  pages={2399--2423},
  year={2023},
  doi={10.3934/dcdss.2023104}
}

@article{he2018time,
  title={Time-varying economic dominance in financial markets: A bistable dynamics approach},
  author={He, Xue-Zhong and Li, Kai and Wang, Chuncheng},
  journal={Chaos: An Interdisciplinary Journal of Nonlinear Science},
  volume={28},
  number={5},
  pages={055903},
  year={2018},
  doi={10.1063/1.5021141}
}

@article{caginalp2020nonlinear,
  title={Nonlinear price dynamics of S\&P 100 stocks},
  author={Caginalp, Gunduz and DeSantis, Mark},
  journal={Physica A: Statistical Mechanics and its Applications},
  volume={547},
  pages={122067},
  year={2020},
  publisher={10.1016/j.physa.2019.122067}
}

@article{dieci2018steady,
  title={Steady states, stability and bifurcations in multi-asset market models},
  author={Dieci, Roberto and Schmitt, Noemi and Westerhoff, Frank},
  journal={Decisions in Economics and Finance},
  volume={41},
  number={2},
  pages={357--378},
  year={2018},
  publisher={Springer},
doi={10.1007/s10203-018-0214-3}
}
\end{document}